\documentclass[a4paper,11pt]{article}
\usepackage[T1]{fontenc} 
\usepackage{amsmath,amsthm}
\usepackage[francais,english]{babel}
\usepackage[dvips]{graphicx}
\usepackage{todonotes}
\usepackage{color}
\usepackage{xcolor}
\usepackage{tikz}
\usetikzlibrary{arrows,positioning,shapes,matrix,backgrounds}
\usetikzlibrary{fit,calc,decorations.pathreplacing}

\usepgflibrary{arrows}
\usepackage{soul}
\usepackage{amsfonts,amssymb}
\usepackage{geometry}
\usepackage{hyperref}
\usepackage{stmaryrd}
\usepackage{fancyhdr}
\usepackage{url}
\usepackage{todonotes}
\usepackage{dsfont}
\usepackage[utf8]{inputenc} 
\usepackage{enumerate}

\theoremstyle{definition}

\newtheorem*{thm*}{Theorem}

\newtheorem{prop}{Proposition}[section]

\newtheorem{lemma}[prop]{Lemma}

\newtheorem{thm}[prop]{Theorem}
\newtheorem{definition}[prop]{Definition}

\newtheorem{corollary}[prop]{Corollary}

\newtheoremstyle{pourlesremarques}{\topsep}{\topsep}{\normalfont}{}{\bfseries}{.}{ }{}
\theoremstyle{pourlesremarques}
\newtheorem{rem}[prop]{Remark}
\newtheorem*{rem*}{Remark}
\newtheoremstyle{pourlesexemples}{\topsep}{\topsep}{\normalfont}{}{\bfseries}{.}{ }{}
\theoremstyle{pourlesexemples}

\def\presuper#1#2%
  {\mathop{}%
   \mathopen{\vphantom{#2}}^{#1}%
   \kern-\scriptspace%
   #2}
   
\newcommand{\Wh}{\mathcal{W}}

\newcommand{\F}{\mathfrak{\mathcal{F}}}

\newcommand{\Ind}{\operatorname{Ind}}
\newcommand{\ind}{\operatorname{ind}}

\renewcommand{\i}{\iota}
\newcommand{\Hom}{\operatorname{Hom}}
\renewcommand{\subset}{\subseteq}

\newcommand{\G}{\mathcal{G}}

\newcommand{\vol}{\operatorname{vol}}
\newcommand{\GL}{\operatorname{GL}}

\newcommand{\Ker}{\mathrm{Ker}}
\renewcommand{\k}{\kappa}

\renewcommand{\d}{\delta}
\newcommand{\R}{\mathbb{R}}

\newcommand{\C}{\mathbb{C}}

\newcommand{\Q}{\mathbb{Q}}
\newcommand{\N}{\mathbb{N}}
\newcommand{\M}{\mathcal{M}}

\renewcommand{\S}{\mathcal{C}_c^\infty}

\newcommand{\e}{\epsilon}
\newcommand{\1}{\mathbf{1}}

\def\GL{\operatorname{GL}}
\def\Id{\operatorname{Id}}

\def\\Hom{\operatorname{\Hom}}

\def\dim{\operatorname{dim}}

\def\Ql{\overline{\mathbb{Q}_{\ell}}}

\def\Zl{\overline{\mathbb{Z}_{\ell}}}

\def\leq{\leqslant}

\def\L{\mathcal{L}}

\def\r{\mathfrak{r}}
\def\OI{\Omega_{\mathrm{I}}}
\def\E{\mathcal{E}}

\newcommand{\D}{\Delta}

\def\presuper#1#2%
  {\mathop{}%
   \mathopen{\vphantom{#2}}^{#1}%
   \kern-\scriptspace%
   #2}
\makeatletter
\DeclareRobustCommand{\rvdots}{%
  \vbox{
    \baselineskip4\p@\lineskiplimit\z@
    \kern-\p@
    \hbox{.}\hbox{.}\hbox{.}  \hbox{.}
  }}
\makeatother
\definecolor{dviolet}{RGB}{100,0,100}
\definecolor{blue}{RGB}{0,0,100}

\title{Generalized Whittaker functions and Jacquet modules}
\author{N. Matringe\footnote{Nadir Matringe, Universit\'e de Poitiers, Laboratoire de Math\'ematiques et Applications,
T\'el\'eport 2 - BP 30179, Boulevard Marie et Pierre Curie, 86962, Futuroscope Chasseneuil Cedex. France. \newline Email:~Nadir.Matringe@math.univ-poitiers.fr}}
\date{\today}
\begin{document}
\maketitle

\begin{abstract}
Let $F$ be a non Archimedean local field and $G$ be (the $F$-points of) a connected reductive group defined over $F$. Fix $U_0$ to be the unipotent radical of a minimal parabolic subgroup $P_0$ of $G$, and $\psi:U_0\rightarrow \C^\times$ be a non-degenerate character of $U_0$. Let $P=MU\supseteq P_0$ be a standard parabolic subgroup of $G$ so that the restriction $\psi_M$  of $\psi$ to $M\cap U_0$ is non-degenerate. We denote by $\Wh(G,\psi)$ the space of smooth $\psi$-Whittaker functions on $G$ 
and by $\Wh_c(G,\psi)$ its $G$-stable subspace consisting of functions with compact support modulo $U_0$. In this situation Bushnell and Henniart identified $\Wh_c(M,\psi_M^{-1})$ to the the Jacquet module of $\Wh_c(G,\psi^{-1})$ with respect to $P^-$ (\cite{BHWh}). On the other hand Delorme defined a constant term map from 
$\Wh(G,\psi)$ to $\Wh(M,\psi_M)$ which descends to the Jacquet module of $\Wh(G,\psi)$ with respect to $P$ (\cite{DelormeCT}). We show (as we surprisingly could not find a proof of this statement in the literature) that the descent of Delorme's constant term map is the dual map of the isomorphism of Bushnell and Henniart, in particular the constant term map is surjective. We also show that the constant term map coincides on admissible submodules of $\Wh(G,\psi)$ with the inflation of the "germ map" defined by Lapid and Mao (\cite{LM}) following earlier works of Casselman and Shalika (\cite{CS}). From these results we derive a simple proof of a slight generalization of a theorem of Delorme (\cite{DelormePWi}) and Sakellaridis--Venkatesh (\cite{SV} for quasi-split $G$) on irreducible discrete series with a generalized Whittaker model to the setting of admissible representations with a central character under the split component of $G$, and similar statements in the cuspidal case (also generalizing a result of Delorme) and in the tempered case. We also show that the germ map of Lapid and Mao is injective, answering one of their questions. Finally using a result of Vignéras from \cite{V2} and recent results from \cite{DHKM22}, we show in the context $\ell$-adic representations that the asymptotic expansion of Lapid and Mao can be chosen to be integral for functions in integral $G$-submodules of $\Wh(\pi,\psi)$ of finite length. 
\end{abstract}

\section{Introduction}

Whittaker functions and their various generalizations play an important role in the representation theory of $p$-adic and real groups, in particular due to the fact that they appear naturally as a local analogue of Fourier coefficients of automorphic forms.
In this paper $F$ is a non Archimedean local field and $G$ be (the $F$-points of) a connected reductive group defined over $F$. We fix $U_0$ the unipotent radical of a minimal parabolic subgroup $P_0$ of $G$, and $\psi:U_0\rightarrow \C^\times$ be a non-degenerate character of $U_0$. Let $P=MU\supseteq P_0$ be a standard parabolic subgroup of $G$ so that the restriction $\psi_M$  of $\psi$ to $M\cap U_0$ is non-degenerate. We denote by $\Wh(G,\psi)$ the space of smooth $\psi$-Whittaker functions on $G$ 
and by $\Wh_c(G,\psi)$ its $G$-stable subspace consisting of functions with compact support modulo $U_0$. In this situation Bushnell and Henniart identified $\Wh_c(M,\psi_M^{-1})$ to the the Jacquet module of $\Wh_c(G,\psi^{-1})$ with respect to $P^-$ the parabolic subgroup of $G$ such that $P\cap P^-=M$ (\cite[2.2 Theorem]{BHWh}, see Section \ref{section BH iso}). On the other hand Delorme defined a constant term map from 
$\Wh(G,\psi)$ to $\Wh(M,\psi_M)$ which descends to the Jacquet module of $\Wh(G,\psi)$ with respect to $P$ (\cite[Definition 3.12]{DelormeCT}, see Section \ref{section constant term map}). Both maps have simple characterizations as shown by the authors of \cite{BHWh} and \cite{DelormeCT} repsectively, and this allows us to compare them without too much effort. In fact we prove in Theorem \ref{theorem BH=D} (as we could not find a proof of this statement in the literature, though it might be known to experts) that the descent of Delorme's constant term map is the dual map of the isomorphism of Bushnell and Henniart, in particular the constant term map is surjective and its descent identifies the Jacquet module of $\Wh(G,\psi)$ with respect to $P$ to $\Wh(M,\psi_M)$. We also show in Theorem \ref{theorem Xi vs D} that the constant term map coincides on admissible submodules of $\Wh(G,\psi)$ with the inflation of the "germ map" defined by Lapid and Mao in \cite{LM} following earlier works of Casselman and Shalika from \cite{CS}. From these results we derive in Theorem \ref{theorem quick DSV} a very quick proof of a generalization of a theorem of Delorme (\cite[Théorème 9]{DelormePWi}) and Sakellaridis--Venkatesh (\cite[Corollary 6.3.5]{SV} where $G$ is assumed to be quasi-split) on irreducible discrete series with a generalized Whittaker model to the setting of admissible representations with a central character under the split component of $G$. We get similar statements in the cuspidal and tempered case, generalizing \cite[Théorème 10]{DelormePWi} in the cuspidal case. Theorem \ref{theorem Xi vs D} also shows that the germ map of Lapid and Mao is injective, answering \cite[Question 3.2]{LM}. After restating the asymptotic expansion of Lapid and Mao for functions in admissible submodules of $\Wh(\pi,\psi)$ (Theorem \ref{theorem asymptotic expansion}), we finally show in Theorem \ref{theorem integral asymptotic expansion} in the context $\ell$-adic representations, that the asymptotic expansion in question can be chosen to be integral for integral $G$-submodules of $\Wh(\pi,\psi)$ of finite length.  This last result makes crucial use of results of Vignéras from \cite{V2}  and of the very recent work \cite{DHKM22} (see Section \ref{section V}).

\textbf{Acknowledgement}. We are very grateful to R. Kurinczuk for pointing out the role of the Bushnell--Henniart isomorphism to us in the context of this paper. We thank E. Lapid for useful correspondence. We also thank the referees 
for their very useful comments, allowing in particular to correct some mistakes and clarify some arguments.

\section{Notations}

\subsection{Reductive groups}

Here $F$ is a non Archimedean local field with residual characteristic $p$ and normalized absolute value $|\ |$. We denote by 
$G$ the $F$-points of a connected reductive group defined over $F$. We fix a minimal parabolic subgroup $P_0$ of $G$, denote by 
$U_0$ its unipotent radical, and fix a Levi subgroup $M_0$ of $P_0$. We denote by $A_0$ 
the $F$-split component of (the center of) $M_0$, the group $M_0/A_0$ is compact. The letter $P$ will always denote a 
standard parabolic subgroup of $G$ (i.e. $P_0\subset P$), we denote by $M$ its standard Levi subgroup (i.e. $M$ is the Levi component of $P$ containing $A_0$) and by $U$ its unipotent radical. We will denote by $A_M$ the split component of 
$M$, and by $A_M^1$ the maximal compact (open) subgroup of $A_M$. We denote by $P^-$ the parabolic subgroup opposite to $P$ with respect to $A_0$, so that $P\cap P^-=M$ and we denote by 
$U^-$ its unipotent radical, so that $P^-=MU^-$.
 We denote by 
$\D(A_M,P)_G$ or most of the time by $\D(A_M,P)$ (resp. $\D(A_M,P^-)$) the set of simple roots of $A_M$ for its action on the Lie algebra of $P$ (resp. $P^-$). We set $\D_0=\D(A_0,P_0)$, and when $P$ is a standard parabolic subgroup of $G$, we set 
\[\D_0^P=\D(A_0,M\cap P_0)_M\subset \D_0.\] 
 
For $\e>0$ we set 
\[A_M^-(P,\e)=\{a\in A_M, \ |\alpha(a)|\leq\e, \alpha \in \D(A_M,P)\}\]
and 
\[A_M^-(P^-,\e)=\{a\in A_M, \ |\alpha(a)|\leq\e, \alpha \in \D(A_M,P^-)\}\] so that 
\[a\in A_M^-(P,\e)\Longleftrightarrow a^{-1}\in A_M^-(P^-,\e).\]
Note that $A_M^-(P,1)\subset A_0^-(P_0,1)$. 
For $0<\e\leq 1$ we set \[A_{0,P}^-(\e)=\{a\in A_0, \ |\alpha(a)|\leq\e \ \mathrm{for} \ \alpha \in \D_0-\D_0^P\ \mathrm{and} 
\ \e<|\alpha(a)|\leq 1 \ \mathrm{for}\ \alpha \in \D_0^P \}.\] 

We denote by $K_0$ a maximal compact open subgroup in good position with respect to $(P_0,P_0^-)$ (see \cite[V.2.1 and V.5.2]{Renard} where the terminlogy there is "$K_0$ is adapted to $A_0$", in particular the Iwasawa decomposition $G=P_0K$ holds) with its family 
$\OI(G)$ of open subgroups $K$ having an Iwahori decomposition with respect to $(P,P^-)$ for any standard parabolic subgroup $P$ of $G$, and such that $K\cap M$ has the same property with respect to standard parabolic subgroups of $M$ (containing $P_0\cap M$). It is known that $\OI(G)$ form a basis of neighborhoods of the identity in $G$. For $K\in \OI(G)$ and $P=MU$ a standard parabolic subgroup of $G$ we set $K_U=K\cap U$, $K_{U^-}=K\cap U^-$ and $K_M=K\cap M$ so that 
$K=K_{U}K_MK_{U^-}$. For any 
$a\in  A_M^-(P,1)$ we have $a^{-1}K_{U^-}a\subset K_{U^-}$ and $aK_{U^-}a^{-1}\subset K_{U^-}$.

\subsection{Haar measures}
If $X$ a totally compact locally disconnected space, we denote by $\mathcal{C}_c^\infty(X)$ the space of (complex valued) compactly supported locally constant functions on $X$. If moreover $X$ is a group we denote by $\mathcal{C}^\infty(X)$ the space of smooth functions on $X$ (i.e. functions which are fixed on the right by a compact open subgroup of $X$) and 
$\mathcal{C}_c^\infty(X)\subset \mathcal{C}^\infty(X)$. We fix a right Haar measure $dh$ on any closed subgroup $H$ of $G$ normalized by the condition $\mu(H\cap K_0)=1$. We denote by $\d_H:H\rightarrow \R_{>0}$ the modulus character of $H$, which is determined by the fact that $\d_H dh$ is a left Haar measure on $H$. If $P'\subset P$ are two standard parabolic subgroups of $G$ then $\d_{P'}\equiv \d_P$ on $A_M$. Such choices fix a unique right invariant measure on $U\backslash G$ such that 
\[\int_G f(g) dg=\int_{U\backslash G} \int_U f(ug) du dg\] for any 
$f\in \mathcal{C}_c^\infty(G)$.
The following integration formulas are valid for any 
$f\in \mathcal{C}_c^\infty(U\backslash G)$ and we will use them freely:
\[\int_{U\backslash G} f(g) dg= \int_{M}\int_{K_0} f(mk)\d_P(m)^{-1} dk dm \] and 
\[\int_{U\backslash G} f(g) dg= \int_{M}\int_{U^-} f(mu^-)\d_P(m)^{-1} dm du^-.\]

\subsection{Smooth representations}\label{section smooth reps}

We say that $(\pi,V)$ is a smooth (complex) representation of $G$ if every $v\in V$ is fixed by some 
$K \in \OI(G)$. We say that it is moreover admissible if $V^K$ has finite dimension 
for all (equivalently some) $K\in \OI(G)$. If $\pi$ is a smooth representation of $G$, we denote by $\pi^\vee$ its smooth dual. 
We denote by $\psi$ a non-degenerate character of $U_0$ (\cite[1.2.Definition]{BHWh}), and if $P=MU\supseteq P_0$ we denote by 
$\psi_M$ its restriction to $U_0\cap M$ which is non-degenerate as well by \cite[2.2. Proposition]{BHWh}. We denote by \[(\rho_{G,\psi}^\vee,\Ind_{U_0}^G(\psi)) \ \mathrm{or} \ (\rho_{G,\psi}^\vee,\Wh(G,\psi))\] the representation by right translation of $G$ on the space of smooth (i.e. fixed on the right by some $K\in \OI(G)$) functions from $W:G\rightarrow \C$ satisfying $W(ug)=\psi(u)W(g)$ for 
$u\in U_0, \ g\in G$. We denote by \[(\rho_{G,\psi},\ind_{U_0}^G(\psi^{-1}))  \ \mathrm{or} \ (\rho_{G,\psi},\Wh_c(G,\psi^{-1}))\] the restriction of $\rho_{G,\psi^{-1}}^\vee$ to the $G$-invariant subspace of functions with compact support modulo $U_0$. This strange notation makes sense: the duality 
 \[\langle W',W \rangle_G=\int_{U_0\backslash G} W'(g)W(g)dg\] for $W'\in \ind_{U_0}^G(\psi^{-1})$ and $W\in \Ind_{U_0}^G(\psi)$ identifies 
 $\rho_{G,\psi}^\vee$ to $(\rho_{G,\psi})^\vee$. We will now omit the subscript $\psi$ in $\rho_{G,\psi}$ and most of the time also omit the subscript $G$ in $\rho_G$. We say that $\xi$ in the algebraic dual $V^*$ of a smooth representation $(\pi,V)$ of $G$ 
 is a $\psi$-Whittaker functional if \[\xi(\pi(u_0)v)=\psi(u_0)\xi(v)\] for $u_0\in U_0,\ v\in V$. For example the map 
 $\d_e:\Wh(G,\psi)\rightarrow \C$ defined by $W\mapsto W(e)$ (where $e$ is the neutral element of $G$) is a $\psi$-Whittaker functional on $\Wh(G,\psi)$ which restricts non trivially to any nonzero $G$-submodule of $\Wh(G,\psi)$.

\section{Preliminary results of Bushnell--Henniart and Delorme}

\subsection{Delorme's application of the second adjunction theorem}

 We recall here the second adjointness theorem for smooth representations due to J. Bernstein (which generalizes that of W. Casselman for admissible representations). We then explain a consequence of it due to P. Delorme. For $(\pi,V)$ a smooth (complex) representation of $G$ and $P=MU$ a standard parabolic subgroup of $G$, we denote by $(r_P(\pi),r_P(V))$ its non normalized Jacquet module and by $(J_P(\pi),J_P(V))$ its normalized Jacquet module (the second being the smooth $M$-module obtained by twisting the first by $\d_P^{-1/2}$). 
We recall that Bernstein proved (following Jacquet and Casselman in the admissible case) that if $(\pi,V)$ is a smooth representation of $G$ and if $K\in \OI(G)$, then the map $r_P$ sends $V^K$ to $r_P(V)^{K\cap M}$ surjectively (see 
\cite[VI.9.1 Théorème]{Renard}). More precisely we have Bernstein's second adjointness theorem as stated in \cite[Lemma 2.1]{DelormeCT} (in \cite[Lemma 2.1]{DelormeCT} the quantity $\e_K$ below is taken uniformly with respect to all smooth representations, but we don't need to know that this can be done). 

\begin{thm}\label{theorem Bernstein form 1}
Let $P$ be a standard parabolic subgroup of $G$ and $K\in \OI(G)$. Let $(\pi,V)$ be a smooth representation of $G$, then $r_P(\pi)^\vee\simeq r_{P-}(\pi^\vee)$ for any smooth representation $\pi$ of $G$. Moreover one can take the $M$-invariant duality $\langle\ , \ \rangle_P$ between $r_{P-}(\pi^\vee)$ and $r_P(\pi)$ identifying  $r_{P-}(\pi^\vee)$ to $r_{P}(\pi)^\vee$ to satisfy the following property:  there exists $0<\e_{K}<1$ such that for $(v,v^\vee)\in V^K\times (V^\vee)^K$ then  
\begin{equation}\label{equation bernstein duality 1} \langle r_P(\pi)(a)r_P(v),r_{P^-}(v^\vee)\rangle_P=\langle\pi(a)v,v^\vee\rangle \end{equation} for 
all $a\in A_M^-(P,\e_{K})$.
The duality $\langle \ , \ \rangle_P$ is uniquely characterized by $\langle \ , \ \rangle$ and the existence of $\e_K$ such that 
(\ref{equation bernstein duality 1}) holds for all $a\in A_M^-(P,\e_{K})$ for every $K\in \OI(G)$.
 If moreover $\pi$ is admissible, then $\langle \ , \ \rangle_P$ is already uniquely characterized by $\langle \ , \ \rangle$ and the fact that for $(v,v^\vee)\in V\times V^\vee$, there exists $0<\e<1$ (depending on $(v,v^\vee)$) such that 
(\ref{equation bernstein duality 1}) holds for all $a\in A_M^-(P,\e)$.
\end{thm}

\begin{rem}\label{rem A} The fact that $\langle \ , \ \rangle$ together with (\ref{equation bernstein duality 1}) characterize of the duality $\langle \ , \ \rangle_P$ follows from the following facts: for $K\in \OI(G)$ one has $r_P(V^K)=r_P(V)^{K\cap M}$, $r_{P^-}((V^\vee)^K)=r_{P^-}(V^\vee)^{K\cap M}$ and $r_P(\pi)(a)$ is a linear bijection of $r_P(V)^{K\cap M}$ for any $a\in A_M^-(P,\e_K)$ (one suffices). Together they imply that $\langle \ , \ \rangle_P$ is uniquely determined on $r_P(V)^{K\cap M}\times r_{P^-}(V^\vee)^{K\cap M}$ for all $K\in \OI(G)$, hence on $r_P(V)\times r_{P^-}(V^\vee)$. When $\pi$ is admissible then the weaker assumption gives the existence of an $\e_K$ by considering the minimal value of the positive numbers $\e$ obtained for the vectors in a basis of $V^K\times (V^\vee)^K$, hence is enough to characterize the duality $\langle \ , \ \rangle_P$. \end{rem}

It will be convenient to reformulate Theorem \ref{theorem Bernstein form 1} as follows.

\begin{corollary}\label{corollary Bernstein form 2}
Let $P$ be a standard parabolic subgroup of $G$ and let $(\pi,V)$ be a smooth representation of $G$, then for all $K\in \OI(G)$ there exists $\e_K>0$ such that 
\begin{equation}\label{equation bernstein duality 2} \langle r_{P^-}(v),r_P(\pi^\vee(a))r_{P}(v^\vee)\rangle_{P^-}=\langle v,\pi^\vee(a)v^\vee\rangle \end{equation} for all $(v,v^\vee)\in V^K\times (V^\vee)^K$ and  
$a\in A_M^-(P,\e_K)$.
The duality $\langle \ , \ \rangle_{P^-}$ is uniquely characterized by $\langle \ , \ \rangle$ and  
(\ref{equation bernstein duality 2}) for all $K\in \OI(G)$. When $\pi$ is admissible the duality $\langle \ , \ \rangle_{P^-}$ is already uniquely characterized by $\langle \ , \ \rangle$ and the fact that for $(v,v^\vee)\in V\times V^\vee$, there exists $0<\e<1$ (depending on $(v,v^\vee)$) such that 
(\ref{equation bernstein duality 2}) holds for all $a\in A_M^-(P,\e)$.
\end{corollary}

Delorme proves in \cite[Theorem 3.4]{DelormeCT} a more general form of the following result.

\begin{thm}\label{theorem Jacquet module of Whittaker functionals}
Let $P=MU$ be a standard parabolic subgroup of $G$. Let $(\pi,V)$ be a smooth representation of $G$ and let $\xi\in V^*$ a $\psi$-Whittaker functional. Then there is a unique $\psi_M$-Whittaker functional $r_{P^-}(\xi)$ in 
$r_{P^-}(V)^*$ such that for any $K\in \OI(G)$, there exists $0<\e_K<1$ such that 
\begin{equation}\label{equation jacquet module of WF} \langle r_P(\pi)(a)r_P(v),r_{P^-}(\xi)\rangle=\langle \pi(a)v, \xi \rangle\end{equation} for all 
$a\in A_M^-(P,\e_K)$ and $v\in V^K$. When $\pi$ is admissible, the map $r_{P^-}(\xi)$ is already uniquely determined by the existence for any $v\in V$, of $0<\e<1$ such that (\ref{equation jacquet module of WF}) holds for all $a\in A_M^-(P,\e)$
\end{thm}

\begin{rem}
Uniqueness of $r_{P^-}(\xi)$ again follows from the facts that $r_P(V^K)=V^{K\cap M}$ and that $r_P(\pi)(a)$ is a linear bijection of $r_P(V)^{K\cap M}$ for any $a\in A_M^-(P,\e_K)$. In the admissible case the weaker assumption characterizes the duality by the same argument as in Remark \ref{rem A}.
\end{rem}

\textit{From now on, for every $K\in \OI(G)$, we fix $0<\e_K<1$ such that the statements of Corollary \ref{corollary Bernstein form 2} holds for $V=\Wh(G,\psi)$ with this choice of $\e_K$, and that of Theorem \ref{theorem Jacquet module of Whittaker functionals} also holds with this choice of $\e_K$. (We recall once again, though we shall not use it, that by \cite[Lemma 2.1]{DelormeCT} there is a choice of $\e_K$ which holds for all smooth representations together).}

\subsection{The constant term map}\label{section constant term map}

Thanks to Theorem \ref{theorem Jacquet module of Whittaker functionals} Delorme defined in \cite[Definition 3.12]{DelormeCT} a constant term map from 
$\Wh(G,\psi)$ to $\Wh(M,\psi_M)$. 

\begin{definition}
For $P=MU$ a standard parabolic subgroup of $G$ and $W\in \Wh(G,\psi)$, we set 
\[W_P(m)=\d_P^{-1/2}(m)\langle r_P(\rho^\vee(m)W),r_{P^-}(\d_e)\rangle\] for $m\in M$. 
\end{definition}

The map $W_P$ belongs to $\Wh(M,\psi_M)$ and $D_P:W\mapsto W_P$ is a $P$-module homomorphism from $(\Wh(G,\psi),\d_P^{-1/2}\rho^\vee)$ to $\Wh(M,\psi_M)$ where the action of $U$ (resp. $M$) on $\Wh(M,\psi_M)$ is trivial (resp. given by $\rho_M^\vee$), in other words $D_P$ induces an $M$-module homomorphism 
\[\overline{D}_P:J_P(\Wh(G,\psi))\rightarrow \Wh(M,\psi).\] We call $W_P$ the constant term of $W$ along $P$. Delorme then shows in \cite[Proposition 3.14]{DelormeCT} that 
the constant term map $D_P$ is characterized as follows. We recall his proof in view of Corollary \ref{corollary admissible subs of W} below.

\begin{prop}\label{proposition characterization of the constant term}
The constant term map $D_P:\Wh(G,\psi)\rightarrow\Wh(M,\psi_M)$ is characterized by the property that it is the unique 
$P$-module homomorphism from $(\d_P^{-1/2}\rho_G^\vee,\Wh(G,\psi))$ to $\Wh(M,\psi_M)$ for $U$ acting trivially and $M$ by $\rho_M^\vee$ on 
 $\Wh(M,\psi_M)$, such that for any $K\in \OI(G)$ we have 
\begin{equation}\label{equation characterization of the constant term} \d_P(a)^{1/2}W_P(a)=W(a)\end{equation} for all $a\in A_M^-(P,\e_K)$ and $W\in \Wh(G,\psi)^K$.
\end{prop}
\begin{proof}
Suppose that $D':\Wh(G,\psi)\rightarrow \Wh(M,\psi_M)$ also has the properties which we expect to characterize $D_P$. Then the map $W\mapsto D'(W)(e)$ factors through a map $\psi_M$-Whittaker functional $\eta\in r_P(V)^*$: 
\[D'(W)(e)=\langle r_P(W),\eta \rangle\] for all $W\in \Wh(G,\psi)$. Then for fixed $K\in \OI(G)$ and $W\in \Wh(G,\psi)^K$, Equation (\ref{equation characterization of the constant term}) for $D'(W)$ tells that \[\langle \rho^\vee(a)W, \d_e \rangle=W(a)=\d_P(a)^{1/2}D'(W)(a)= (\d_P(a)^{1/2}\rho_M^\vee(a)D'(W))(e)\] for 
$a\in A_M^-(P,\e_K)$. However by the $M$-equivariance property of $D'$ 
\[(\d_P(a)^{1/2}\rho_M^\vee(a)D'(W))(e)=D'(\rho^\vee(a)W)(e)=\langle r_P(\rho^\vee(a)W), \eta \rangle.\] Hence by Theorem \ref{theorem Jacquet module of Whittaker functionals} we have $\eta=j_{P^-}(\d_e)$, so \[D'(W)(e)=D_P(W)(e)\] for  all $W\in \Wh(G,\psi)$. As 
both maps $D'$ and $D_P$ have the same equivariance property under $M$, we deduce that $D'(W)=D_P(W)$ for all 
$W\in \Wh(G,\psi)$.
\end{proof}

\begin{rem}
Of course in the statement of Proposition \ref{proposition characterization of the constant term} we can take $K$ varying in any subset of $\OI(G)$ which is still a basis of neighborhoods of the identity in $G$ (because we can do so in the statements of Theorem \ref{theorem Bernstein form 1}, Corollary \ref{corollary Bernstein form 2} and Theorem \ref{theorem Jacquet module of Whittaker functionals}), which we will do in the proof of Theorem \ref{theorem BH=D}.
\end{rem}

We end this section by giving a version of Proposition \ref{proposition characterization of the constant term} for admissible submodules of $\Wh(G,\psi)$.

\begin{corollary}\label{corollary admissible subs of W}
Let $V$ be an admissible submodule of $\Wh(G,\psi)$. The retriction to $V$ of the constant term map $D_P$ is characterized by the property that it is the unique 
$P$-module homomorphism from $(\d_P^{-1/2}\rho_G^\vee,V)$ to $\Wh(M,\psi_M)$ for $U$ acting trivially and $M$ by $\rho_M^\vee$ on 
 $\Wh(M,\psi_M)$, such for any $W\in V$, there is $\e>0$ such that 
\begin{equation}\label{equation characterization of the constant term} \d_P(a)^{1/2}W_P(a)=W(a)\end{equation} for $a\in A_M^-(P,\e)$.
\end{corollary}
\begin{proof}
The proof of Proposition \ref{proposition characterization of the constant term} goes through with $V$ instead of the full representation $\Wh(G,\psi)$, except that we appeal to the admissible version of Theorem \ref{theorem Jacquet module of Whittaker functionals} at the end of it.
\end{proof}

\subsection{An isomorphism of Bushnell--Henniart}\label{section BH iso}

We recall from the introduction that $\Wh_c(G,\psi)$ is the $G$-stable subspace of $\Wh(G,\psi)$ with consisting of functions 
compactly supported modulo $U_0$, and from Section \ref{section smooth reps} that we denote by $\rho$ the action of $G$ on $\Wh_c(G,\psi)$ by right translation. In \cite{BHWh}, for $P=MU$ a standard parabolic subgroup of $G$, Bushnell and Henniart identify the Jacquet module of $\Wh_c(G,\psi)$ with respect to $P^-$ to $\Wh_c(M,\psi_M)$.  We recall their result (\cite[2.2 Theorem]{BHWh}), which moreover gives a characterization of the map $W\mapsto J_{P^-}(W)$ via their identification.

\begin{thm}\label{theorem BH isomorphism}

Let $P=MU$ be a standard parabolic subgroup of $G$, and turn $\Wh_c(M,\psi_M)$ into a $P^-$-module by making 
$U^-$ act trivially on $\Wh_c(M,\psi_M)$ and $M$ act by $\rho_M$. There is a unique $P^-$-module homomorphism map 
$BH_{P^-}:(\Wh_c(G,\psi^{-1}),\d_P^{1/2}\rho)\rightarrow \Wh_c(M,\psi_M^{-1})$ such that if $U_c^-$ is a compact open subgroup of $U^-$ and 
$W\in \Wh_c(G,\psi^{-1})^{U_c^-}$ has support in $PU_c^-$, then 
\[BH_{P^-}(W)=\vol(U_c^-)\d_{P}^{-1/2}W_{|M}.\]
The map $BH_{P^-}$ is surjective and induces a $M$-module isomorphism 
\[\overline{BH}_{P^-}:J_{P^-}(W)\mapsto BH_{P^-}(W)\] between $J_{P^-}(\Wh_c(G,\psi^{-1}))$ and
$\Wh_c(M,\psi_M^{-1})$.
\end{thm}

\section{Surjectivity of the constant term map}

We now have the map $BH_{P^-}:\Wh_c(G,\psi)\rightarrow \Wh_c(M,\psi_M)$ which induces 
an $M$-module isomorphism $\overline{BH}_{P^-}:J_{P^-}(\rho)\simeq \rho_M$ and the constant term map 
$D_P:\Wh(G,\psi)\rightarrow \Wh(M,\psi_M)$ which induces 
an $M$-module homomorphism $\overline{D}_{P}:J_{P}(\rho^\vee)\rightarrow \rho_M^\vee.$ 
Denote by \[\overline{BH}_{P^-}^\vee: \rho_M^\vee \simeq (J_{P^-}\rho)^\vee\] the $M$-module isomorphism 
dual to $\overline{BH}_{P^-}$, and by 
 \[\overline{BH}^P: \rho_M^\vee \simeq J_{P}(\rho^\vee)\] the isomorphism obtained by composing 
 $\overline{BH}_{P^-}^\vee$ with the isomorphism of the second adjunction theorem (Theorem \ref{theorem Bernstein form 1}). 
 Finally we set \[\overline{BH}_P:=(\overline{BH}^P)^{-1}: J_{P}(\rho^\vee)\simeq \rho_M^\vee \] and write \[BH_P:\d_P^{-1/2}\rho^\vee \twoheadrightarrow \rho_M^\vee\] the surjection obtained by inflating $\overline{BH}_P$.  
 We will now show the equality \[D_P=BH_P,\] which in particular implies that $D_P$ is surjective.

We recall that $(\rho^\vee,\Wh(G,\psi))$ is identified to the dual of $(\rho,\Wh_c(G,\psi^{-1}))$ thanks to the duality 
\[\langle W', W \rangle_G= \int_{U_0\backslash G} W'(g)W(g)dg\] for $(W',W)\in \Wh_c(G,\psi^{-1})\times \Wh(G,\psi)$ and that 
 $(\rho_M^\vee,\Wh(M,\psi_M))$ is identified to the dual of $(\rho_M,\Wh_c(M,\psi_M^{-1}))$ thanks to the duality 
\[\langle W_M', W_M \rangle_M= \int_{M\cap U_0\backslash M} W_M'(m)W_M(m)dm\] for $(W_M',W_M)\in \Wh_c(M,\psi_M^{-1})\times \Wh(M,\psi_M)$. In particular through our various identifications we have for any  $(W',W)\in \Wh_c(G,\psi^{-1})\times \Wh(G,\psi)$ with respective image $(\overline{W'},\overline{W})\in J_{P^-}(\Wh_c(G,\psi^{-1})\times J_P(\Wh(G,\psi^{-1}))$:
\begin{equation}\label{equation step 1} \langle \overline{W'}, \overline{W} \rangle_{P^-}= 
\langle \overline{BH}_{P^-}(\overline{W'}), \overline{BH}_P(\overline{W})\rangle_M=
\langle BH_{P^-}(W'), BH_P (W) \rangle_M.\end{equation} On the other hand 
Corollary \ref{corollary Bernstein form 2} tells us that given $K\in \OI(G)$, for $K$-invariant $W'$ and $W$ one has  
\begin{equation}\label{equation step 2} \langle \overline{W'}, J_P(\rho^\vee(a))\overline{W} \rangle_{P^-}=\d_P(a)^{-1/2}\langle W', \rho^\vee(a)W \rangle_G \end{equation} for 
all $a\in A_M^-(P,\e_K)$. So take $K\in \OI(G)$ small enough such that $K_{U_0}\subset \Ker(\psi)$, take $W\in \Wh(G,\psi)^K$, and define $W_0 \in \Wh_c(G,\psi^{-1})^K$ to be function supported on $U_0K=U_0K_M K_{U^-}$ equal to $1$ on $K$. Putting Equations (\ref{equation step 1}) and (\ref{equation step 2}) together we obtain for $a\in A_M^-(P,\e_K)$
\[\d_P^{1/2}(a) \langle BH_{P^-}(W_0), \rho_M^\vee(a) BH_P (W)\rangle_M= \langle BH_{P^-}(W_0), BH_P (\rho^\vee(a)W)\rangle_M\]
\[= \langle W_0, \rho^\vee(a)W \rangle_G =\int_{U_0 \backslash G} W_0(g)W(ga) dg\]
\[= \int_{U_0\cap M \backslash M}\int_{U^-} W_0(mu^-)W(m u^{-} a)\d_P^{-1}(m) d m d {u^-} \]
\[= \int_{U_0\cap M \backslash M}\int_{U^-} W_0(mu^-)W(m a (a^{-1} u^{-} a))\d_P^{-1}(m) d m d {u^-}\]
\[=\vol(K_{U^-})\int_{U_0\cap M \backslash M}W_0(m)W(m a )\d_P^{-1}(m) d m\]
\[=\vol(K_{U^-})\int_{M\cap U_0^-}\int_{A_0}W_0(a_0v_0^-)W(a_0 v_0^- a )\d_{M\cap P_0}^{-1}(a_0)\d_P^{-1}(a_0) da_0d{v_0^-}\]
\[=\vol(K_{U^-})\int_{M\cap U_0^-}\int_{A_0}W_0(a_0 v_0^-)W(a_0 a (a^{-1}v_0^- a))\d_{M\cap P_0}^{-1}(a_0)\d_P^{-1}(a_0) da_0d{v_0^-}\]
\[=\vol(K_{U^-})\vol(K_M \cap U_0^-)\int_{A_0}W_0(a_0)W(a_0 a)\d_{M\cap P_0}^{-1}(a_0)\d_P^{-1}(a_0) da_0 \]
\[=\vol(K_{U^-})\vol(K_M \cap U_0^-)\int_{K\cap A_0}W(a a_0)\d_{M\cap P_0}^{-1}(a_0)\d_P^{-1}(a_0) da_0 \]
\[=\vol(K_{U^-})\vol(K_M \cap U_0^-)\vol(K\cap A_0)W(a).\]

On the other hand 
\[ \d_P^{1/2}(a) \langle BH_{P^-}(W_0), \rho_M^\vee(a) BH_P (W)\rangle_M\]
\[=\d_P^{1/2}(a) \vol(K_{U^-})\int_{U_0\cap M \backslash M}W_0(m)BH_P(W)(m a )\d_P^{-1/2}(m)d m\]
\[=\d_P^{1/2}(a) \vol(K_{U^-})\vol(K_M \cap U_0^-)\vol(K\cap A_0)BH_P(W)(a)\] because $BH_P(W)$ is right $K\cap M$-invariant as $W$ is 
right $K$-invariant. From the above discussion we obtain the equality 
\[W(a)=\d_P^{1/2}(a) BH_P(W)(a)\] for all $W\in \Wh(G,\psi)^K$ and $a\in A_M^-(P,\e_K)$. As moreover both $BH_P$ and $D_P$ have the same equivariance property with respect to $P$ by definition of $BH_P$, we arrive at the following conclusion thanks to Proposition \ref{proposition characterization of the constant term}:

\begin{thm}\label{theorem BH=D}
The maps $BH_P$ and $D_P$ from $\Wh(G,\psi)$ to $\Wh(M,\psi_M)$ are equal and in particular $D_P$ is surjective.
\end{thm}

\section{Further notations and preliminaries}

Let $P$ be a standard parabolic subgroup of $G$. For $\alpha \in \D_0$, we set $F_\alpha=F-\{0\}$ if 
$\alpha \in \D_0^P$ and $F_\alpha=F$ otherwise.
We set \[\M_P^G=\prod_{\alpha\in \D_0} F_\alpha\] and recall that \[\mathcal{C}_c^\infty(\M_P^G)\simeq \bigotimes_{\alpha\in \D_0} \mathcal{C}_c^\infty(F_\alpha).\]
We denote by $\r_G:A_0\rightarrow \prod_{\alpha\in \D_0} F-\{0\} \subset \M_P$ the map defined by $\r_G(t)_\alpha=\alpha(t)$. 

Because the restriction map from the lattice $X^*(M)$ of algebraic characters of $M$ to that $X^*(A_M)$ of algebraic characters of $A_M$ is injective with finite co-kernel, any $|\chi|$ for $\chi\in X^*(A_M)$ extends uniquely to 
a positive character of $M$ still denoted by $|\chi|$. 
For $\e>0$ we set 
\[M^-(P,\e)=\{m\in M,\ |\alpha|(m)\leq \e,\ \alpha\in \D(A_M,P)\}.\]

Whenever $V$ is a $\C$-vector space affording a smooth representation $\pi$ of $A_M$, we denote by 
$V_{A_M,\mathrm{fin}}$ its subspace of $A_M$-finite vectors, i.e. the space of vectors $v$ in $V$ such that 
$\C[A_M].v$ is finite dimensional. For example when $V=\mathcal{C}^\infty(A_M)$ with $\pi:=\rho$ the action by translation, 
we set \[\F(A_M):=\mathcal{C}^\infty(A_M)_{A_M,\mathrm{fin}}\]
Note that 
\[V_{A_M,\mathrm{fin}}=\bigoplus_{\chi\in \widehat{A_M}} V_{(\chi)}\] where 
\[V_{(\chi)}=\{v\in V, \ \exists n\in \N-\{0\}, (\pi(a)-\chi(a)\Id)^n v=0\ \forall a\in A_M\}.\]
The space $V_{(\chi)}$ is an $A_M$-submodule of $V_{A_M,\mathrm{fin}}$ which we call the characteristic subspace associated to $\chi$, and which contains the eigenspace 
\[V_{\chi}=\{v\in V, \ (\pi(a)-\chi(a)\Id) v=0\ \forall a\in A_M\}.\] We denote by $\widehat{A_M}$ the group of smooth characters from $A_M$ to $\C^*$ and set 
\[\E(A_M,V)=\{\chi\in \widehat{A_M}, \ V_{(\chi)}\neq \{0\}\}=\{\chi\in \widehat{A_M}, \ V_{\chi}\neq \{0\}\}.\]
When $V$ is an admissible representation of $M$ then $V=V_{A_M,\mathrm{fin}}$. 
We recall from \cite[Section 4]{Casselman} that if $V$ is an admissible representation of $G$, then $J_P(V)$ is an admissible representation of $M$ for $P=MU$ any standard parabolic subgroup of $G$. 

\section{The germ map}\label{section germ}

Here we recall ideas of Casselman and Shalika (\cite{CS}) further developed by Lapid and Mao in \cite{LM}. 
One says that two functions $f$ and $f'$ in $\mathcal{C}^\infty(A_M)$ (resp. $\mathcal{C}^\infty(M)$)  have the same germ if there is $\e>0$ such that 
$f_{|A_M^-(P,\e)}={f'}_{|A_M^-(P,\e)}$ (resp. $f_{|M^-(P,\e)}={f'}_{|M^-(P,\e)}$) and we write $f\sim f'$. This defines an equivalence relation on $\mathcal{C}^\infty(A_M)$ (resp. $\mathcal{C}^\infty(M)$) and we denote by $[f]$ the class of $f$ and call it its germ. We set $\G(A_M)=[\mathcal{C}^\infty(A_M)]$ (resp. $\G(M)=[\mathcal{C}^\infty(M)]$). The map 
$f\mapsto [f]$ from $\mathcal{C}^\infty(A_M)$ to $\G(A_M)$ (resp. from $\mathcal{C}^\infty(M)$ to $\G(M)$) is a smooth $A_M$-bimodules (resp. $M$-bimodules) homomorphism. We note that the paper \cite{LM} is valid for $F$ of positive characteristic as well (the characteristic zero assumption plays no role in the paper). By \cite[Lemma 2.9]{LM}, the map $f\mapsto [f]$ induces an isomorphism $\Gamma_M$ between $\mathcal{C}^\infty(A_M)_{A_M,\mathrm{fin}}$ and $\G(A_M)_{A_M,\mathrm{fin}}$. From this one deduces as in the 
proof of \cite[Corollary 2.11]{LM} that $f\mapsto [f]$ induces an isomorphism 
\[\i_M:\mathcal{C}^\infty(M)_{A_M,\mathrm{fin}}\simeq \G(M)_{A_M,\mathrm{fin}}.\]

It is then proved in the proof of \cite[Theorem 3.1]{LM} that if $\pi\subset \Wh(G,\psi)$ is an admissible $G$-module (Lapid and Mao work under a finite length assumption which is unnecessary), then the map 
\[W\mapsto [\d_P^{-1/2}W_{|M}]\] factors through a map \[\k_P: J_P(\pi)\rightarrow \G(M).\] Moreover 
as \[J_P(\pi)=J_P(\pi)_{A_M,\mathrm{fin}}\] because $J_P(\pi)$ is admissible, we deduce that 
\[\k_P: J_P(\pi)\rightarrow \G(M)_{A_M,\mathrm{fin}}.\] We then define 
\[\Xi_P:=\i_M^{-1}\circ \k_P: J_P(\pi)\rightarrow \mathcal{C}^\infty(M)_{A_M,\mathrm{fin}}.\] The authors of 
\cite{LM} note that in fact \[\Xi_P: J_P(\pi)\rightarrow \Wh(M,\psi_M)_{A_M,\mathrm{fin}}.\] Here we answer 
\cite[Question 3.2]{LM}.

\begin{thm}\label{theorem Xi vs D}
Let $\pi\subset \Wh(G,\psi)$ an admissible $G$-module, then the "germ map" $\Xi_P$ coincides with 
$\overline{D_P}=\overline{BH}_P:J_P(W)\mapsto W_P$ on $\pi$. In particular it is injective hence $\Ker(\k_P)=\{0\}$.
\end{thm}
\begin{proof}
Thanks to the characterization of $D_P$ given in Corollary \ref{corollary admissible subs of W}, it is sufficient to show that there exists $\e>0$ such that such that 
$\Xi_P(W)=\d_P^{-1/2}(a)W(a)$ on $A^-(P,\e)$. However by definition $\Xi_P(W)$ and $\d_P^{-1/2}W_{|M}$ have the same germ, hence 
in particular they agree on $A_M^-(P,\e)\subset M^-(P,\e)$ for some $\e>0$.
\end{proof}

\section{The asymptotic expansion of Whittaker functions}

The following expansion is due to Lapid and Mao (\cite[Theorem 3.1]{LM}), we reproduce parts of their proof which holds verbatim in our more general context, as we shall need to use them later in the context of integral $\ell$-adic representations. 

\begin{thm}\label{theorem asymptotic expansion}
Let $(\pi,V)$ be an admissible $\C[G]$-submodule of $\Wh(G,\psi)$ and $W\in V$, then for each $P=MU \supseteq P_0$ and $\chi\in \E(A_M,J_P(\pi))$, there is a finite indexing set $I_{W,\chi}$ (where we choose $I_{W,\chi}$ disjoint from 
$I_{W,\chi'}$ if $\chi\neq \chi'$) such that one can write for $t\in A_0$ 
\[W(t)=\sum_{P=MU \supseteq P_0} \d_P(t)^{1/2} \sum_{\chi\in \E(A_M,J_P(\pi))}\sum_{i\in I_{W,\chi}} f_i(t)\phi_i(\r_G(t))\]
for some functions $f_i\in \F(A_0)_{(\chi)}$ and $\phi_i\in \S(\M_P^G)$.
\end{thm}
\begin{proof}
We do an induction on the split rank $r$ of $G/A_G$. If $r=0$ then $\S(\M_G^G)$ is just the space of constant functions on 
$A_0=A_G$. Take $W\in V$, because $V$ is admissible $W$ is $A_G$-finite and the expansion follows. If $r>0$ we can always suppose that $W\in V_{(\mu)}$ for $\mu\in \E(A_G,\pi)$. For $\emptyset\neq J\subset \D_0$ we denote by $P_J=M_J U_J$ the standard parabolic subgroup of $G$ such that $\D_0-\D_0^{P_J}=J$. By induction 
we have for any subset $\emptyset\neq J\subset \D_0$: 
\[W_{P_J}(t)= \sum_{P\subset P_J} \d_{P\cap M_J}(t)^{1/2} \sum_{\chi\in \E(A_M,J_P(\pi))}F_{J,\chi}(t)\] where 
\[F_{J,\chi}(t)=\sum_{k\in I_{W,J,\chi}} f_k(t)\phi_k(\r_{M_J}(t))\] for some finite indexing set $I_{W,J,\chi}$ (with all such sets disjoint from one another when $J$ and $\chi$ vary) and $f_k\in \F(A_0)_{(\chi)}$, $\phi_k\in \S(\M_{P\cap M_J}^{M_J})$. Hence arguing as in the proof of \cite[Theorem 3.1]{LM} there is $\e>0$ such that for all subsets $\emptyset\neq J\subset \D_0$ one has 
\[W(t)=\sum_{P\subseteq P_J} \d_{P}(t)^{1/2} \sum_{\chi\in \E(A_M,J_P(\pi))}F_{J,\chi}(t)\] on 
$A_0^-(P_0,\epsilon)$. Setting 
\[F_\chi(t)= \sum_{\emptyset \neq J \subset \D_0-\D_0^P}(-1)^{|J|-1}F_{J,\chi}(t)\prod_{\alpha\in J} \mathbf{1}_{]0,\e[}(|\alpha(t)|)\] we observe that 
\[F_{J,\chi}(t)\prod_{\alpha\in J} \mathbf{1}_{]0,\e[}(|\alpha(t)|)
=\sum_{k\in I_{W,J,\chi}} f_k(t)\phi'_k(\r_G(t))\] 
where \[\phi'_k(\r_G(t))=\phi_k(\r_{M_J}(t))\prod_{\alpha\in J}\mathbf{1}_{]0,\e[}(|\alpha(t)|)\in \S(\M_{P}^{G}).\] Then Lapid and Mao show that 
\[W(t)=\sum_{P=MU\subsetneq G} \sum_{\chi\in \E(A_M,J_P(\pi))} \d_P^{1/2}(t)F_\chi(t)\] for any $t\in A_0$ such that there exists 
$\alpha\in \D_0$ with $|\alpha(t)|<\e$, whereas the sum on the right vanishes for $t$ not satisfying this property. But then 
we claim that the function \[\Phi(t)=W(t)\prod_{\alpha\in \D_0}\mathbf{1}_{[\e,+\infty[}(|\alpha(t)|)\] belongs to 
$\F(A_0)_{(\mu)}\S(\M_G^G)$ and the equality \[W(t)=\sum_{P=MU\subsetneq G} \sum_{\chi\in \E(A_M,J_P(\pi))} \d_P^{1/2}(t)F_\chi(t)+\Phi(t)\] will then end the proof of the theorem. Indeed using the equivariance property of $W$ under $U_0$ we have $W(t)=0$ whenever $|\alpha(t)|$ is large enough for some 
$\alpha\in \D_0$, hence there is $C$ such that \[\Phi(t)=W(t)\prod_{\alpha\in \D_0}\mathbf{1}_{[\e,C]}(|\alpha(t)|).\] Now 
$W$ is finite under $A_G$ because $V$ is admissible and the result follows. 
\end{proof}

\section{Applications}

In this section we generalize to admissible representations of $G$ 
on which $A_G$ acts by a character a theorem of Delorme (\cite[Proposition 13 and Théorème 9]{DelormePWi}) and Sakellaridis--Venkatesh (\cite[Corollary 6.3.5]{SV} for quasi-split groups) for irreducible representations, characterizing discrete series with a generalized Whittaker model. We also prove similar statements for cuspidal and tempered representations, generalizing \cite[Théorème 10]{DelormePWi} in the cuspidal case.

\begin{lemma}\label{lemma eigenvector in the Jacquet module}
Let $\pi\subset \Wh(G,\psi)$ be an admissible $G$-submodule. Let $P=MU$ be a standard parabolic subgroup of $G$ (possibly $G$ itself) such that $\E(A_M,J_P(\pi))\neq \emptyset$ and take $\chi\in \E(A_M,J_P(\pi))$. Then there is $W\in \pi$ and $\e>0$ such that 
$W(t)=\d_P^{1/2}(t)\chi(t)$ on $A_M^-(P,\e)$. 
\end{lemma}
\begin{proof}
 By Theorem \ref{theorem BH=D} and the discussion before it the map $W\mapsto W_P$ on $\Wh(G,\psi)$ identifies with the normalized Jacquet projection $J_P$, hence the same holds on any $G$-submodule of $\Wh(G,\psi)$, for example $\pi$. In particular the space $W_P$ for $W\in \pi$ is the normalized Jacquet module of $\pi$ with respect to $P$. Now take $W\in \pi$ such that $W_P\neq 0$ in the $\chi$-eigenspace of $J_P(\pi)$. Up to translating $W$ by an element of $M$, and normalizing it by a nonzero scalar in $\C^\times$, we can suppose that 
$W_P(1)=1$, which in turn implies that $W_P(t)=(\rho(t)W_P)(1)=\chi(t)$ for all $t\in A_M$. On the other hand the asymptotic expansion of Theorem \ref{theorem asymptotic expansion} gives the existence of $\e>0$ and distinct characters $\chi_1,\dots,\chi_r$ of $A_M$ such that one has on $A_M^-(P,\e)$ an expansion of the form $W(t)=\sum_{i=1}^r f_i(t)$ with $f_i\in \F(A_M)_{(\chi_i)}$. Hence according to Corollary 3.9 we have $W_P(t)=\sum_{i=1}^r \d_P(t)^{-1/2}f_i(t)$ on $A_M^-(P,\e)$ up to taking $\e$ even smaller. We recall that as already noticed at the beginning of Section \ref{section germ}, there is a unique $f\in \mathcal{C}^\infty(A_M)_{A_M,\mathrm{fin}}$ such that $[f]=[(W_P)_{|A_M}]$, and this $f$ has to be at the same time $\sum_{i=1}^r \d_P^{-1/2}f_i$ and $\chi$, so that only one $f_i$ (say $f_1$) is nonzero and $\d_P^{-1/2}f_1=\chi$. This last equality gives the seeked equality  $W(t)=\d_P(t)^{1/2}\chi(t)$ on $A_M^-(P,\e)$, where $W\in \pi$ by assumption.
\end{proof}

Let $(\pi,V)$ be an admissible representation of $G$. We recall that $\pi$ is called cuspidal if all its Jacquet modules associated to proper standard parabolic subgroups are zero. If $A_G$ acts on $V$ by a unitary character, then we say that $\pi$ is a discrete series if the module of every matrix coefficient of $\pi$ belongs to $L^2(A_G\backslash G)$. Thanks to Casselman, this is known to be equivalent to 
the fact that for every proper standard parabolic subgroup $P=MU$ of $G$, every $\chi\in \E(A_M,J_P(V))$ satisfies 
$|\chi(t)|<1$ for $t\in A_M^-(P,1)\cap (A_0^1A_G)^c$, we say that $\chi$ is \textit{positive} (\cite[Theorem 4.4.6]{Casselman}). Similarly still under the hypothesis that $A_G$ acts acts on $V$ by a unitary character, then we say that $\pi$ is tempered if the module of every matrix coefficient of $\pi$ belongs to $L^{2+r}(A_G\backslash G)$ for any $r>0$. One shows as in 
\cite[Theorem 4.4.6]{Casselman} that it is  equivalent to the fact that for every proper standard parabolic subgroup $P=MU$ of $G$, every $\chi\in \E(A_M,J_P(V))$ satisfies 
$|\chi(t)|\leq 1$ for $t\in A_M^-(P,1)$, we say that $\chi$ is \textit{non negative}. In particular discrete series are tempered. If $c$ is a character of $A_G$, we denote by $\mathcal{C}^\infty(A_G U_0\backslash G, c\otimes \psi)$ the set of smooth functions $W$ from $G$ to $\C$ which satisfy $W(au_0g)=c(a)\psi(u_0)W(g)$ for $a\in A_G,\ u_0\in U_0, \ g\in G$. We denote by  $\mathcal{C}_c^\infty(A_G U_0\backslash G, c\otimes \psi)$ its $G$-submodule consisting of functions with support compact modulo $A_GU_0$. If $c$ is moreover unitary, for $p>0$ we denote by $L^p(A_G U_0\backslash G, c\otimes \psi)^\infty$ the $G$-submodule of $\mathcal{C}^\infty(A_G U_0\backslash G, c\otimes \psi)$ consisting of functions $W$ such that $|W|^p$ is integrable on $A_G U_0\backslash G$.

\begin{thm}\label{theorem quick DSV}
Let $\pi$ be an admissible submodule of $\Wh(G,\psi)$ such that $A_G$ acts on $\pi$ by a unitary character $c$, then $\pi$ is cuspidal, resp. square integrable, resp. tempered if and only if $\pi\subset \mathcal{C}_c^\infty(A_G U_0\backslash G, c\otimes \psi)$, resp. $\pi\subset L^2(A_G U_0\backslash G, c\otimes \psi)^\infty$, resp. $\pi \subset L^{2+r}(A_G U_0\backslash G, c\otimes \psi)^\infty$ for all $r>0$. In the cuspidal case the unitary assumption on $c$ is superfluous.
\end{thm}
\begin{proof}
One direction of all statements easily follows from Theorem \ref{theorem asymptotic expansion} and the Iwasawa decomposition 
$G=U_0M_0 K_0$ together with the compactness of the quotient $M_0/A_0$. We give the proof of the converse direction in the cuspidal and the tempered case, the square integrable case being similar to the tempered case. Suppose that 
$\pi\subset \mathcal{C}_c^\infty(A_G U_0\backslash G, c\otimes \psi)$ and let $P=MU$ be a proper standard parabolic subgroup of $G$. Suppose that $J_P(\pi)\neq \{0\}$ so that $\E(A_M,J_P(\pi))\neq \emptyset$ and take $\chi$ inside it. The function $W$ given by Lemma \ref{lemma eigenvector in the Jacquet module} can't belong to $\mathcal{C}_c^\infty(A_G U_0\backslash G, c\otimes \psi)$, hence $J_P(\pi)$ must be equal to $\{0\}$ and $\pi$ is cuspidal. If $\pi \subset L^{2+r}(A_G U_0\backslash G, c\otimes \psi)^\infty$ for all $r>0$ then
\[\int_{A_G\backslash A_0} |W(t)|^{2+r}\d_{P_0}^{-1}(t) dt<+\infty\] for all $W\in \pi$ thanks to the Iwasawa decomposition. In particular for all standard parabolic subgroups $P$, the function $F =W \d_{P_0}^{-1/(2+r)}$ must satisfy that $|F|^{2+r}$ is summable on $A_G K_{A_0}\backslash A_{0,P}^-(\e)$ for some $0<\e\leq 1$ and some compact open subgroup $K_{A_0}$ of $A_0$. Now take $P=MU$ a proper standard parabolic subgroup of $G$ such $J_P(\pi)\neq \{0\}$ if it exists (if not 
$\pi$ is cuspidal hence tempered). Take $\chi\in \E(A_M,J_P(\pi))$ and $W$ as in Lemma \ref{lemma eigenvector in the Jacquet module}. First we notice that thanks to the asymptotic expansion of theorem \ref{theorem asymptotic expansion}, for $\e$ small enough, the function 
$F$ restricts to $A_{0,P}^-(\e)$ as an $A_M$-finite function. Then by our choice of $W$, the character $\d_P^{1/2}\chi$ is associated to 
$W_{|A_{0,P}^-(\e)}$ in the sense of \cite[Proposition 4.4.4]{Casselman}, hence 
$\d_P^{1/2}\chi\d_{P_0}^{-1/(2+r)}=\d_P^{r/2(2+r)}\chi$ (as $\d_{P_0}$ coincides with $\d_P$ on $A_M$) is associated to $F_{|A_{0,P}^-(\e)}$. We conclude by \cite[Proposition 4.4.4]{Casselman} that $\d_{P}^{-r/2(2+r)}\chi$ must be positive for all $r>0$, hence 
that $\chi$ must be non negative. We conclude that $\pi$ is tempered.
\end{proof}

\section{Asymptotics of integral $\ell$-adic Whittaker functions}

In this last section we consider for $\ell\neq p$ a prime number the field $\Ql$ (a fixed algebraic closure of $\mathbb{Q}_\ell$) instead of $\C$. We fix an isomorphism between $\Ql$ and $\C$ which we use to transport smooth representations of any closed subgroup $H$ of $G$ over $\C$ to smooth representations of $H$ over $\Ql$, this in particular applies the modulus character 
$\d_H$. We denote by $\Zl$ the ring of integers of $\Ql$, and by $\Q_{\ell}^u$ the maximal unramified extension of $\Q_{\ell}$. We denote by $\Q_{\ell}(p^\infty)$ the algebraic extension of $\Ql$ obtained by adjoining all roots of unity of order a power of $p$, hence 
$\Q_{\ell}(p^\infty)\subset \Q_{\ell}^u$. If $E\subset \Ql$ is an extension of $\Q_\ell$ we denote by $O_E$ its ring of integers ($O_E=E\cap \Zl$). We note that if $E$ is contained in a finite extension of $\Q_{\ell}^u$, then $O_E$ is principal. 

\subsection{A result of Vignéras on integral submodules of $\Wh(G,\psi)$}\label{section V}

Let $(\pi,V)$ be an admissible representation of $G$, following \cite[I. Definition 9.1]{Vbook} we say that 
a $G$-stable $\Zl$-submodule $\L$ of $V$ is an admissible lattice if $\L^K$ is a $\Zl$-lattice in $V^K$ (i.e. 
free of rank $\dim_{\Ql}(V^K)$) for all $K\in \OI(G)$. Let $E\subset \Ql$ be an algebraic extension of $\Q_{\ell}$. We say that a smooth ($\ell$-adic) representation $(\pi,V)$ is realizable over $E$ if it has an $E$-structure $(\pi_E,V_E)$: the $E$-vector space $V_E\subset V$ is $G$-stable and $V_E\otimes_E \Ql= V$. If $E$ is contained in a finite extension of $\Q_{\ell}^u$ and $(\pi,V)$ is admissible and has an $E$-structure, then $(\pi,V)$ is integral if and only if there is a $G$-stable $O_E$-lattice $\L_E$ in $V_E$. Indeed it follows from \cite[I.9.2]{Vbook} that if $\L_E$ is such a lattice, then $\L=\L_E\otimes_{O_E} \Zl$ is an admissible lattice in $V$, and conversely if $\L$ is an admissible lattice in $V$ then $\L_E:=\L\cap V_E$ is an admissible lattice in $V_E$ as we now justify: for any $x\in V$ there is $m_x\in \N$ large enough such that $\ell^{m_x} x\in \L$ from which we deduce that for any $K\in \OI(G)$ we the $E$-span of $\L_E^K$ is $V_E^K$, and moreover $\L_E$ being torsion free and $E$ being principal we deduce that $\L_E^K$ is a free $O_E$-module. 
By \cite[II.4.7]{Vbook} any irreducible representation $\pi$ of $G$ is realizable over a finite extension of 
$\Q_{\ell}$, and the proof of this fact extends to finite length representations, we add it here.

\begin{prop}\label{proposition finite length realizable}
Let $(\pi,V)$ be an $\ell$-adic representation of finite length, then it is realizable over a finite extension of 
$\Q_{\ell}$.
\end{prop}
\begin{proof}
We suppose that $V\neq \{0\}$. Take $\{0\}=V_0 \subsetneq \dots \subsetneq V_r=V$ a composition series of $V$ of length $r$. Take $K\in \OI(G)$ small enough such that each $(V_{i+1}/V_i)^K$ is nonzero. Then one checks by induction on the length of 
$V$ that $\pi(G).V^K=V$. Denote by $\mathcal{H}(G,K)$ the Hecke algebra of bi-$K$-invariant functions inside $\mathcal{C}_c^\infty(G)$, it is well-known that the category of smooth representations spanned over $G$ by their $K$-invariants is 
equivalent to that of $\mathcal{H}(G,K)$-modules (\cite[II. 3.12]{Vbook} for example). Now $\mathcal{H}(G,K)$ is finitely generated (see \cite[II. 2.13]{Vbook}), hence the image of $\mathcal{H}$ of $\mathcal{H}(G,K)$ inside $\mathrm{End}_{\Ql}(V^K)$ is finitely generated. Take $B=(e_1,\dots,e_n)$ a basis of $V^K$, then $\mathrm{Mat}_{B}(\mathcal{H})$ is a finitely generated 
subalgebra of $\M_n(\Ql)$, hence is contained in $\M_n(E)$ for $E$ a finite extension of $\Q_\ell$. Set $W_E=\mathrm{Vect}_E(e_1,\dots,e_n)$, then $W_E$ is a finite length $\mathcal{H}(G,K)_E$-module (where $\mathcal{H}(G,K)_E$ consists of 
functions in $\mathcal{H}(G,K)$ with values in $E$). Then by \cite[II. 3.12]{Vbook} there is a unique $E[G]$-module 
of finite length $V_E$ such that $V_E^K=W_E$. One checks that $V_E$ is an $E$-structure for $V$. 
\end{proof}

We notice that the non degenerate character $\psi:U_0\rightarrow \Ql$ takes values in $O_{\Q_{\ell}(p^\infty)}$. Whenever $R$ is a subring of $\Ql$ which contains $O_{\Q_{\ell}(p^\infty)}$, we denote by $\Wh(G,\psi)_R$ the $R$-submodule of 
$\Wh(G,\psi)$ of functions with values in $R$ and set $\Wh_c(G,\psi^{-1})_R=\Wh_c(G,\psi^{-1})\cap \Wh(G,\psi^{-1})_R$. Combining \cite[Theorem IV.2.1, Corollary II.8.3, Proposition II.8.2]{V2} of Vignéras we obtain the following result.

\begin{prop}\label{proposition V}
Let $E$ be an algebraic extension of $\Q_{\ell}(p^\infty)$ contained in a finite extension of $\Q_{\ell}^u$ and $(\pi_E,V_E)\subset (\Wh(G,\psi)_E,\rho^\vee)$ be an admissible representation of $G$ which is integral. Then the $O_E[G]$-module \[\L_E:=V_E\cap \Wh(G,\psi)_{O_E}\] is an $O_E$-lattice in $V_E$. Moreover $(\pi_E^\vee,V_E^\vee)$ is a quotient of $(\Wh_c(G,\psi^{-1})_E,\rho)$ and if we denote by $s:W'\in \Wh_c(G,\psi^{-1})_E \mapsto \overline{W'}\in V_E^\vee$ the surjection such that the injection $i:V_E \subset  \Wh(G,\psi)_E$ is dual to it, then $\L_E':=\overline{\Wh_c(G,\psi^{-1})_{O_E}}=s(\Wh_c(G,\psi^{-1})_{O_E})$ is an $O_E[G]$-lattice in $\pi_E^\vee$ which identifies $\L_E$ with the dual lattice of $\L_E'$: 
\[\L_E=\{W\in V_E,\ \langle \Wh_c(G,\psi^{-1})_{O_E}, W\rangle \subset O_E\}=\{W\in V_E,\ \langle \L_E', W\rangle \subset O_E\}.\] 
\end{prop}

Note that with the above notations, if $E'$ is a finite extension of $E$, we have the inclusion $\L'_E\subset \L'_{E'}$ (in fact $ \L'_{E'}=\L'_E \otimes_{O_{E}} O_{E'}$). Indeed set $s':W'\in \Wh_c(G,\psi^{-1})_{E'} \mapsto \overline{W'}\in V_{E'}^\vee$ which we recall is by definition the surjection such that the injection  $i':V_{E'} \subset  \Wh(G,\psi)_{E'}$ is dual to it, then $s'$ restricts to $ \Wh_c(G,\psi^{-1})_E$ as $s$, or equivalently $s'=s \otimes_E {E'}$, because $i'$ also satisfies the relation 
$i'=i\otimes_E {E'}$. The claim follows. 

As a corollary we obtain:

\begin{corollary}\label{corollary extension of lattices}
Let $E$ be an algebraic extension of $\Q_{\ell}(p^\infty)$ contained in a finite extension of $\Q_{\ell}^u$ and $(\pi_E,V_E)\subset (\Wh(G,\psi)_E,\rho^\vee)$ be an admissible representation of $G$ which is integral. 
Let $E'$ be a finite extension of $E$ and $V_{E'}=V_E\otimes _E E'$, clearly $(\pi_{E'},V_{E'})\subset (\Wh(G,\psi)_{E'},\rho^\vee)$. Then $\L_{E'}:=V_{E'}\cap \Wh(G,\psi)_{O_{E'}}$ is equal to $\L_E\otimes_{O_E} O_{E'}$ for $\L_E$ as in Proposition \ref{proposition V}.
\end{corollary}
\begin{proof}
The inclusion $\L_E\otimes_{O_E} O_{E'}\subset \L_{E'}$ is obvious. To prove the converse take $W\in \L_{E'}$, and in fact take $W\in \L_{E'}^K$ for $K\in \OI(G)$ small enough. Because 
$\L_E^{K}$ is a lattice in $V_E^{K}$ any $O_E$-basis $(W_i)_i$ of $\L_E^K$ is an $E'$-basis of $V_{E'}^K$ hence $W=\sum_i a_i W_i$ for $a_i\in E'$. Because $\L_E^K=\{W\in V_E^K, \ \langle {\L_E'}^K, W\rangle \subset O_E \}$ is the dual lattice of ${\L_E'}^K$, the vectors $(\overline{W'_i})_i$ forming the dual basis of $(W_i)_i$ belong to ${\L_E'}^K$. Hence $a_i=\langle \overline{W_i'},W\rangle \in O_{E'}$ because $\overline{W_i'}\in \L'_E\subseteq \L_{E'}'$ as observed before the corollary. 
\end{proof}

The following result now follows.

\begin{thm}\label{theorem integral structure of Whittaker}
Let $(\pi,V)\subset (W(G,\psi),\rho^\vee)$ be an $\ell$-adic admissible representation realizable over an algebraic extension of $\Q_{\ell}$ contained in a finite extension of $\Q_{\ell}^u$. If $\pi$ is integral, then the $\Zl[G]$-module $V\cap W(G,\psi)_{\Zl}$ is an admissible lattice in $V$.
\end{thm}
\begin{proof}
Call $E$ the extension of the statement and set $L=\langle \Q_{\ell}(p^\infty),E \rangle$, it is an extension of $\Q_{\ell}(p^\infty)$ contained in a finite extension of $\Q_{\ell}^u$. Then because $\pi$ has an $E$ structure, it has an $L$-structure and Proposition \ref{proposition V} shows that $\L_L:=V_L\cap W(G,\psi)_{O_{L}}$ is a $G$-stable $O_L$-lattice in $V_L$. The theorem will follow from the equality $\L_{\Zl}:=V\cap W(G,\psi)_{\Zl}=\L_L\otimes_{O_L} \Zl$. The inclusion $\L_L\otimes_{O_L} \Zl\subset \L_{\Zl}$ is clear. Conversely take $W\in \L_{\Zl}$, because $\L_L$ contains an $L$-basis 
$(W_i)_i$ of $V_L$ which must be a $\Ql$-basis of $V$, we can write $W =\sum_i a_i W_i$ for $a_i\in \Ql$. Set $L'=L((a_i)_i)$, it is a finite extension of 
$L$ because all $a_i$ are algebraic over $\Q_{\ell}$ hence over $L$, and $W$ takes values in $L'$ because all the $W_i$ do. But $W$ also takes values in $\Zl$ by definition, hence in 
$O_{L'}$, so $W\in \L_{L'}$. We deduce from Corollary \ref{corollary extension of lattices} that $W\in \L_L\otimes_{O_L} O_{L'}\subset \L_L\otimes_{O_L} \Zl$ and the result follows.
\end{proof}

We recall that finite length representations of $G$ are admissible (see \cite[VI.6.3 Théorème]{Renard}). From Proposition \ref{proposition finite length realizable} we obtain the following immediate consequence of Theorem \ref{theorem integral structure of Whittaker}.

\begin{corollary}\label{corollary lattice in finite length Whittaker}
Let $(\pi,V)\subset (W(G,\psi),\rho^\vee)$ be an $\ell$-adic representation of finite length. If $\pi$ is integral, then the $\Zl[G]$-module $V\cap W(G,\psi)_{\Zl}$ is an admissible lattice in $V$.
\end{corollary}

\subsection{The asymptotic expansion of integral Whittaker functions}

If $\pi$ is smooth admissible $\Ql[G]$-module of $G$, we say that $\E(A_G,\pi)$ is integral if all its elements take values in $\Zl^\times$. The following lemma is straightforward from \cite[I.9.3]{Vbook}.

\begin{lemma}\label{lemma decomposition of lattices into characteristic subspaces}
Let $V$ be an integral admissible $\Ql[G]$-module with $\Zl[G]$-admissible lattice $\L$. Then $\E(A_G,\pi)$ is integral and 
$\L=\bigoplus_{\chi\in\E(A_G,\pi)} \L_{(\chi)}$, where $\L_{(\chi)}=\L\cap V_{(\chi)}$.
\end{lemma}

For $P=MU$ a standard parabolic subgroup of $G$ and $\chi\in \widehat{A_M}$ with values in $\Zl^\times$, we denote by $\F(A_M)_{(\chi),\Zl}$ the $\Zl$-module of functions in 
$\F(A_M)_{(\chi)}$ taking values in $\Zl$, and by $\S(\M_P^G)_{\Zl}$ that of functions in $\S(\M_P^G)$ taking values in $\Zl$.   Whenever $(\pi,V)$ is an integral finite length $G$-submodule of $(\Wh(G,\psi),\rho^\vee)$, we set 
\[\L_{\pi}=V\cap \Wh(G,\psi)_{\Zl}.\] In particular $\L_{\pi}$ is an lattice in $V$ according to Corollary
\ref{corollary lattice in finite length Whittaker}. First we prove the expected asymptotic expansion when $G/A_G$ has split rank $0$.

\begin{lemma}
Suppose that $G/A_G$ has split rank $0$. Let $V$ be an integral $\Ql[G]$-submodule of $\Wh(G,\psi)$ of finite length and $W\in \L_{\pi}$, then for each $\chi\in \E(A_G,\pi)$, there is a finite set $I_{W,\chi}$ such that one can write for $z\in A_G$ 
\[ W(z)=\sum_{\chi\in \E(A_G,\pi)}\sum_{i\in I_{W,\chi}} f_i(z)\phi_i(\r_G(z))\] 
for some functions $f_i\in \F(A_G)_{(\chi),\Zl}$ and $\phi_i\in \S(\M_G^G)_{\Zl}$ (here $\phi_i$ is just a constant in $\Zl$).
\end{lemma}
\begin{proof}
Note that the $\Ql[G]$-module $\pi':=\Ql[G].W$ has finite dimension because $G/A_G$ is compact and $A_G$ is commutative, hence we can consider a basis $(W_1,\dots,W_r)$ of the $\Zl[G]$-lattice $\L_{\pi'}$. The action of $A_G$ on $\L_{\pi'}$ in this basis has coefficients given by functions in $\F(A_G)_{\Zl}$. Writing $W$ as a sum of the $W_i$'s with integral coefficients, we see that $\pi(z)W$ is of the form $\sum_{i=1}^r h_i(z) W_i$ for $h_i\in \F(A_G)_{\Zl}$, hence evaluating $\pi(z)W$ at $1$ we see that 
$W$ is a linear combination with integral coefficients (the scalars $W_i(1)\in \Zl$) of the functions $h_i$, and we conclude by applying Lemma \ref{lemma decomposition of lattices into characteristic subspaces}.
\end{proof}

It was expected for a long time that whenever $\pi$ is admissible and integral, all its Jacquet modules are integral (see \cite{Dat}). Though this fact sounds elementary it is far from being the case and Dat proved it (with sophisticated arguments) in \cite{Dat} at least when $F$ has characteristic zero and $G$ is $\GL_n(F)$ or an inner form of it, or a classical group. However this statement has been very recently shown to hold in full generality in \cite{DHKM22}, as a consequence of (the very deep results of) \cite{FS} together with some finiteness results on the Galois side of the local Langlands correspondence established in \cite{DHKM22}. This allows us to remove the hypothesis that the Jacquet modules of $\pi$ are integral in the statement below,  which is an integral version of Theorem \ref{theorem asymptotic expansion}.

\begin{thm}\label{theorem integral asymptotic expansion}
Let $\pi$ be an integral $\Ql[G]$-submodule of $W(G,\psi)$ of finite length with integral Jacquet modules. Then $\E(A_M,J_P(\pi))$ is integral for each $P=MU \supseteq P_0$. Moreover for each $P=MU \supseteq P_0$ and $\chi\in \E(A_M,J_P(\pi))$, there is a finite indexing set $I_{W,\chi}$ such that one can write for $t\in A_0$ 
\[W(t)=\sum_{P=MU \supseteq P_0} \d_P(t)^{1/2} \sum_{\chi\in \E(A_M,J_P(\pi))}\sum_{i\in I_{W,\chi}} f_i(t)\phi_i(\r_G(t))\]
for some functions $f_i\in \F(A_0)_{(\chi),\Zl}$ and $\phi_i\in \S(\M_P^G)_{\Zl}$.
\end{thm}
\begin{proof}
First for each $P=MU \supseteq P_0$, the Jacquet module $J_P(\pi)$ has finite length (see \cite[VI.6.4]{Renard}) and it is integral thanks to 
\cite[Corollary 1.5]{DHKM22}, in particular $\E(A_M,J_P(\pi))$ is integral. Moreover 
the map $\Xi_P=\overline{D}_P=\overline{BH}_P$ identifies $J_P(\pi)$ with a submodule of 
$\Wh(M,\psi)$ (Theorem \ref{theorem Xi vs D} for example), in particular $\L_{J_P(\pi)}$ is a lattice in $J_P(\pi)$ thanks to Corollary \ref{corollary lattice in finite length Whittaker}. These observations, together with the fact that functions of the form $\1_{]0,\e[}(|\alpha(\ )|)$, $\1_{[\e,+\infty[}(|\alpha(\ )|)$ and $\1_{[\e,C]}(|\alpha(\ )|)$ have integral values are enough to apply the induction hypothesis going through the proof of Theorem \ref{theorem asymptotic expansion} and to deduce the expected expansion. 
\end{proof}

\bibliographystyle{plain}
\bibliography{Modlfactors}

\end{document}